\documentclass[12pt]{article}

\def\margin{2cm}
\usepackage[top=\margin, bottom=\margin, left=\margin, right=\margin]{geometry}

\usepackage{amsmath, amsthm, amssymb}
\usepackage{pstricks, pst-node}

\usepackage[noblocks]{authblk}

\newtheorem{definition}{Definition}

\newtheorem{theorem}{Theorem}
\newtheorem{lemma}[theorem]{Lemma}

\newtheorem{observation}[theorem]{Observation}

\theoremstyle{remark}

\title{Product Dimension of Forests and\\ Bounded Treewidth Graphs}

\author[1]{L. Sunil Chandran}
\author[2]{Rogers Mathew}
\author[1]{Deepak Rajendraprasad}
\author[1]{Roohani Sharma}
\affil[1]
{
    Department of Computer Science and Automation, \authorcr
    Indian Institute of Science, \authorcr
    Bangalore, India, 560012. \authorcr
    \{sunil, deepakr\}@csa.iisc.ernet.in, \authorcr 
	roohani.sharma90@gmail.com
}
\affil[2]
{
    Department of Mathematics and Statistics, \authorcr
	Dalhousie University, Halifax, Canada - B3H 3J5 \authorcr
    rogersm@mathstat.dal.ca
}
\date{}

\begin{document}
\maketitle

\begin{abstract}
The {\em product dimension} of a graph $G$ is defined as the minimum natural number $l$ such that $G$ is an induced subgraph of a direct product of $l$ complete graphs. In this paper we study the product dimension of forests, bounded treewidth graphs and $k$-degenerate graphs. We show that every forest on $n$ vertices has a product dimension at most $1.441 \log n + 3$. This improves the best known upper bound of $3 \log n$ for the same due to Poljak and Pultr. The technique used in arriving at the above bound is extended and combined with a result on existence of orthogonal Latin squares to show that every graph on $n$ vertices with a treewidth at most $t$ has a product dimension at most $(t+2)(\log n + 1)$. We also show that every $k$-degenerate graph on $n$ vertices has a product dimension at most $\lceil 8.317 k \log n \rceil + 1$. This improves the upper bound of  $32 k \log n$ for the same by Eaton and  R{\"o}dl.

\end{abstract}


\section{Introduction}

For a graph $G(V,E)$ and an $l \in \mathbb{N}$, a function $\phi_{G} : V \rightarrow \mathbb{N}^{l}$ is called an {\em $l$-encoding} of $G$ if 
\begin{enumerate}
 \item $\phi_{G}$ is an injection, and
 \item $\forall u, v \in V, \{u,v\} \in E$ iff $\phi_{G}(u)$ and $\phi_{G}(v)$ differ in all $l$ coordinates.
\end{enumerate}
The minimum $l$ such that an $l$-encoding of $G$ exists is called the \emph{product dimension} of $G$ and is denoted by $pdim(G)$. Some authors refer to it as the {\em Prague dimension} \cite{furedi2000prague}.

The product dimension of a graph $G$ was first defined in \cite{nesetril1978simple} by Ne\v{s}et\v{r}il and R\"{o}dl as the minimum $l$ such that $G$ is an induced subgraph of a direct product (see Section \ref{notDef}) of $l$ complete graphs. It is easy to see that the two definitions of product dimension are equivalent. Another equivalent definition of the product dimension of a graph is the minimum number of proper colorings of $G$ such that any pair of non-adjacent vertices get the same color in at least one of the colorings and not in all of them.
 
The concept of product dimension of a graph was first used to prove the Galvin-Ramsey property of the class of all finite graphs \cite{nesetril1978simple}. Thereafter, this area was separately explored by various people. 

In 1980, Lov\'{a}sz, Ne\v{s}et\v{r}il and Pultr showed that the product dimension of a path on $n+1$ vertices (length $n$) is $\lceil \log n \rceil$ \cite{lovasz1980product}. They also gave a lower bound for the product dimension of a graph (Theorem 5.3 \cite{lovasz1980product}) which in particular tells that the product dimension of a tree on $n$ vertices with $l$ leaves is at least $\log (n - l + 1)$. The authors also suggested that the idea used to encode paths could be extended to study the product dimension of trees. Immediately after this paper, Poljak and Pultr in \cite{poljak1981dimension} came up with bounds on product dimension of trees using the encoding for paths as given in \cite{lovasz1980product}. The results in this paper are $pdim(T) \leq 3 \lceil \log |T| \rceil $ and $\log |m(T)| -1 \leq pdim(T) \leq 3 \lceil \log |m(T)| \rceil + 1$  where, $T$ is a forest and $m(T)$ is the graph obtained from $T$ by recursively deleting a leaf vertex with one or more siblings. In this paper we improve the above upper bound to $1.441 \log |T| + 3$. More recently, in 2010, Ida Kantor in her doctoral thesis \cite{kantor10} determines another upper bound on the product dimension of trees viz. $2 + \lceil \log \delta_{r} \rceil + \sum_{i \in S, 2 \leq i < r} \lceil \log \delta_{i} \rceil + \sum_{i \not \in S, 3 \leq i< r} \lceil \log (\delta_{i}-1) \rceil$, where  $r$ is the radius of the tree, $x$ is a central vertex,  $\delta_{i}$ is the maximum degree among all vertices which are at a distance $r-i$ from $x$ and $S= \{2^{i} : i \in \mathbb{N}\}$. The technique used is a generalization of the technique used by Lov\'{a}sz, Ne\v{s}et\v{r}il and Pultr in \cite{lovasz1980product} for paths.

The product dimension of graphs obtained by amalgamation of smaller graphs was studied in \cite{alles1986dimension}. The idea of using orthogonal Latin squares to encode a disjoint union of complete graphs is given by Evans, Isaak and Narayan in 
\cite{evans2000representations}. This idea is the motivation for our \emph{Amalgamation Lemma for General Graphs} (Lemma \ref{amalgamDegen}) which is a key ingredient for showing that the product dimension of a graph on $n$ vertices with treewidth at most $t$ is at most $(t+2)(\log n +1)$. Orthogonal Latin squares have been known for a long time. In the 1780s Euler demonstrated methods for constructing orthogonal Latin squares of order $t$ where $t$ is odd or a multiple of $4$ and later conjectured that orthogonal Latin squares of order $t \equiv 2$ mod $4$ do not exist. In 1960, Parker, Bose, and Shrikhande in \cite{bose1960further} disproved Euler's conjecture for all $t \geq 10$. Thus, orthogonal Latin squares exist for all orders $t \geq 3$ except $t = 6$. We use this result to prove Lemma \ref{amalgamDegen}.

A parameter closely related to product dimension of a graph $G$ is the equivalence number of the complement of the graph $G$, $\bar{G}$. An \emph{equivalence} is a vertex disjoint union of cliques and the \emph{equivalence number} of a graph $H$ is the minimum number of equivalences required to cover the edges of $H$. In \cite{alon1986covering}, Alon came up with bounds on the equivalence number of a graph showing $\log n - \log d \leq eq(\bar{G}) \leq 2 e^{2} (d+1)^{2} \ln n$, where $G$ is a graph on $n$ vertices with maximum degree $d$.  It is easy to see that $pdim(G) \leq eq(\bar{G}) + 1$ (\cite{cooper10}). Eaton and R\"{o}dl in \cite{eaton1996graphs} proved that $pdim(G) \leq 32 k \log n$ for a $k$-degenerate graph $G$ on $n$ vertices. Since degeneracy of a graph is at most its maximum degree, this result is a significant improvement over Alon's result. We use a probabilistic method to further improve this upper bound to $\lfloor 8.317 k \log n \rfloor + 2$.  

The product dimension of a graph is closely related to the representation number of a graph - a concept introduced by Erd\"{o}s in \cite{erdos1989representations}. A graph $G$ is representable modulo $r$ if there exists an injection $f : V(G) \rightarrow \{0, \ldots, r-1\}$ such that for all $u,v \in V(G)$, $gcd(f(u),f(v))=1$ if and only if $\{u,v\} \in E(G)$. The minimum $r$ modulo which $G$ is representable is called the representation number of $G$. The relationship between the two concepts viz. the product dimension of a graph and representation number of a graph is described in \cite{evans2007representations}.

\subsection{Summary of Results}

\begin{enumerate}
\item For any forest $T$ on $n$ vertices, $pdim(T) \leq 1.441 \log n + 3$ (Theorem \ref{prodDimTrees}).
\item[] This is an improvement over the upper bound for product dimension of trees and forests given by Poljak and Pultr in \cite{poljak1981dimension} viz. $3 \lceil \log n \rceil$. We use a technique of divide and conquer to prove the theorem. The divide operation corresponds to the operation described in our \emph{Splitting Lemma for Forests} (Lemma \ref{lem1}) while the conquer operation corresponds to our \emph{Amalgamation Lemma for Bipartite Graphs} (Lemma \ref{lem2}).

\item For any graph $G$ on $n$ vertices and treewidth $t$, $pdim(G) \leq (t+2)(\log n +1)$ (Theorem \ref{thmTreeWidth}). 
\item[] The techniques used to prove Theorem \ref{prodDimTrees} for trees inspired us to work for graphs with bounded treewidth. Another key ingredient in proving this theorem is the \emph{Amalgamation Lemma for General Graphs} (Lemma \ref{amalgamDegen}) which is based on the existence of orthogonal Latin squares of different orders. Since treewidth $t$ graphs are $t$-degenerate (Section 4.2, \cite{koster02}), it follows from an upper bound on product dimension based on degeneracy of a graph \cite{eaton1996graphs} that $pdim(G) \leq 32 t \log n$. Our result is an improvement over that. 

\item For every $k$-degenerate graph $G$ on $n$ vertices, $pdim(G) \leq \lceil 8.317 k \log n \rceil + 1$ (Theorem \ref{thmDegenerate}).
\item[] We derive this result as an improvement over Eaton's and R\"{o}dl's upper bound of $32 k \log n$ for product dimension of $k$-degenerate graphs \cite{eaton1996graphs}. We use a probabilistic argument to prove the theorem and we believe that our proof is shorter.
\end{enumerate}

\subsection{Notations and Definitions} \label{notDef}
In this paper we consider only undirected, simple, finite graphs. For any graph $G$, $V(G)$ denotes its vertex set and $E(G)$ denotes its edge set. The \emph{cardinality} of a set $S$ is denoted by $|S|$. For a graph $G$, $|G|$ denotes the cardinality of $V(G)$. $N_{G}(u)$ denotes the open neighborhood of vertex $u$ in $G$, i.e. all the vertices adjacent to $u$ in $G$. The \emph{degree} of a vertex $u$, denoted by $d(u)$ is $|N(u)|$. 


For a graph $G$, the \emph{graph induced by a set $X \subset V(G)$}, denoted by $G[X]$, is the graph with $V(G[X]) = X$ and $ E(G[X]) = E(G) \cap \{\{v, v^{'}\}: v, v^{'} \in X\}$.

If $G_{1}$ and $G_{2}$ are two graphs, then $G_{1} \setminus G_{2}$ is the graph $G_{1}[V(G_{1}) \setminus V(G_{2})]$. If $G$ is a graph and $S \subset  V(G)$, then $G \setminus S$ is the graph $G[V(G) \setminus S]$. The \emph{union} of two graphs $G_{1}$ and $G_{2}$, denoted by $G_{1} \cup G_{2}$, is the graph with $V(G_{1} \cup G_{2})=V(G_{1}) \cup V(G_{2})$ and $E(G_{1} \cup G_{2})=E(G_{1}) \cup E(G_{2})$. Moreover, if $V(G_{1}) \cap V(G_{2}) = \phi $, then we call it a \emph{disjoint union} and denote it as $G_{1} \uplus G_{2}$. The \emph{intersection} of two graphs $G_{1}$ and $G_{2}$ is the graph $G_{1} \cap G_{2}$ with $V(G_{1} \cap G_{2}) = V(G_{1}) \cap V(G_{2})$ and $E(G_{1} \cap G_{2}) = E(G_{1}) \cap E(G_{2})$. 

The graph $G_{1} \times G_{2}$ is the \emph{direct product} of two graphs $G_{1}$ and $G_{2}$ with $V(G_{1} \times G_{2})=V(G_{1})   \times  V(G_{2})$ and $E(G_{1}   \times  G_{2})= \{ \{u,v\} : u, v \in V(G_{1})  \times  V(G_{2})$ and if  $u=(x_{1}, x_{2})$, $v=(y_{1},y_{2})$, then $(x_{1},y_{1}) \in E(G_{1})$ and $(x_{2} ,y_{2}) \in E(G_{2}) \}$. 

Let $[n]$ denote the set $\{1, \ldots , n\}$. The set of all natural numbers is denoted by $\mathbb{N}$. $\{a\}^{k}$ denotes the $k$-tuple $(a, \ldots , a)$. Throughout the paper, $\log n$ denotes $\log _{2} n$ and $\ln n$ denotes $\log _{e} n$.

\section{Product Dimension of Forests}
\begin{definition} \label{splittingDefn}
 In a forest $T$ on $n$ vertices, a vertex $v$ is called 
\begin{enumerate}
 \item an \emph{($\epsilon$,2)-split vertex} if $T \setminus \{v\} = T_{1} \uplus T_{2}$ such that $|T_{1}|, |T_{2}| \leq (\frac{1}{2} + \epsilon)n$, and
 \item an \emph{($\epsilon$,3)-split vertex} if $T \setminus \{v\} = T_{1} \uplus T_{2} \uplus T_{3}$ such that $|T_{1}|, |T_{2}|, |T_{3}| \leq (\frac{1}{2} - \epsilon)n$,
\end{enumerate}    
where $T_{1}, T_{2}$ and $T_{3}$ are subgraphs of $T$.
\end{definition}

\begin{lemma} [Splitting Lemma for Forests]\label{lem1}
In every forest $T$, for every $\epsilon \geq 0$, there exists either an \emph{($\epsilon$,2)-split vertex} or an \emph{($\epsilon$,3)-split vertex}.
\end{lemma}

\begin{proof}
Let $n=|T|$. For any $v \in V(T)$, let $C_{1}(v), \ldots ,$ $C_{m}(v)$ denote the (connected) components of $T \setminus \{v\}$ such that $|C_{1}(v)| \geq \cdots \geq |C_{m}(v)|$.

Let us choose $v \in V(T)$ such that $|C_{1}(v)|={\min}\{|C_{1}(u)| : u \in V(T)\}$. First we claim that $|C_{1}(v)| \leq (\frac{1}{2} + \epsilon)n$. For the sake of contradiction, let us assume that $|C_{1}(v)| > (\frac{1}{2} + \epsilon)n$. Let $w \in C_{1}(v) \cap N_{T}(v)$. If $C_{1}(w) \subset C_{1}(v)$, then $|C_{1}(w)|<|C_{1}(v)|$ (because $C_{1}(w) \subset C_{1}(v) \setminus \{v\}$) contradicting the choice of $v$. Hence, $C_{1}(w) \subset T \setminus C_{1}(v)$ and $|C_{1}(w)| \leq n-|C_{1}(v)| < (\frac{1}{2} - \epsilon)n < |C_{1}(v)|$. This again contradicts the choice of $v$.

If $|C_{1}(v)| > (\frac{1}{2} - \epsilon)n$, then $v$ is an \emph{($\epsilon$,2)-split vertex} and $T_{1}=C_{1}(v)$, $T_{2}= T \setminus (T_{1} \cup \{v\})$. Otherwise, let $t=m$ and $F_{1}=C_{1}(v), \ldots , F_{t}=C_{t}(v)$. Hence, $|F_{i}| \leq (\frac{1}{2}- \epsilon)n$ for all $ i \in [t]$. It is easy to see that if $ t \leq 3$, then $v$ is either an \emph{($\epsilon$,3)-split vertex} or an \emph{($\epsilon$,2)-split vertex} with $T_{i}=F_{i}$. If $t \geq 4$, consider a partition $I_{1} \uplus \ldots \uplus I_{k}=[t]$ with minimum possible $k$ such that $|\cup_{j \in I_{l}}F_{j}| \leq (\frac{1}{2} - \epsilon)n$ for all $l \in [k]$. For $k \leq 3$, $v$ is either an \emph{($\epsilon$,2)-split vertex} or an \emph{($\epsilon$,3)-split vertex} with $T_{i}=F_{i}$. Suppose $k \geq 4$, define $F_{l}^{'}=\cup_{j \in I_{l}} F_{j}$, $l \in [k]$ and let $F^{'}$ be the union of smallest two among $\{F_{1}^{'}, \ldots , F_{k}^{'}\}$. Hence, $|F^{'}| \leq \frac{n}{2} \leq (\frac{1}{2} + \epsilon)n$ by the pigeonhole principle. By the minimality in the choice of the partition $I_{1} \uplus \ldots \uplus I_{k}$, $|F^{'}| > (\frac{1}{2} - \epsilon)n$. Thus, $v$ is an \emph{($\epsilon$,2)-split vertex} with $T_{1}=F^{'}$ and $T_{2}=T \setminus (F^{'} \cup\{ v\})$. 
\end{proof}

\begin{definition}
We call an $l$-encoding $\phi _{G}$ of a graph $G$, a \emph{well-begun $l$-encoding} if the first coordinate of $\phi_{G}$ is from $\{0, \ldots, \chi(G)-1\}$.
\end{definition}

\begin{observation} \label{extensionOfEncoding}
For any $q > p$, if $\phi_{G}$ is a $p$-encoding of $G$, then $\psi_{G}$, obtained from $\phi_{G}$ by adding $q-p$ coordinates to $\phi_{G}$ such that for all $p < i \leq q$, the $i$-th coordinate of $\psi_{G}(x)$ is the $p$-th coordinate of $\phi_{G}$, is a $q$-encoding of $G$.
\end{observation}

\begin{lemma}[Amalgamation Lemma for Bipartite Graphs]\label{lem2}
Let $G_{0},\ldots , G_{k-1}$ be bipartite graphs such that $G_{i} \cap G_{j} = \{g\}$ for all $i, j \in \{0, \ldots, k-1\}$, $i \neq j$. Let $G= \cup_{i=0}^{k-1} G_{i}$. For every $i \in \{0, \ldots , k-1\}$, let $\phi_{G_{i}}$ be a \emph{well-begun} $l_{i}$-encoding of $G_{i}$. Then we can construct a well-begun $l$-encoding $\phi_{G}$ of $G$, where $l=max_{0 \leq i \leq k-1}\{l_{i}\} + \lceil \log k \rceil$.

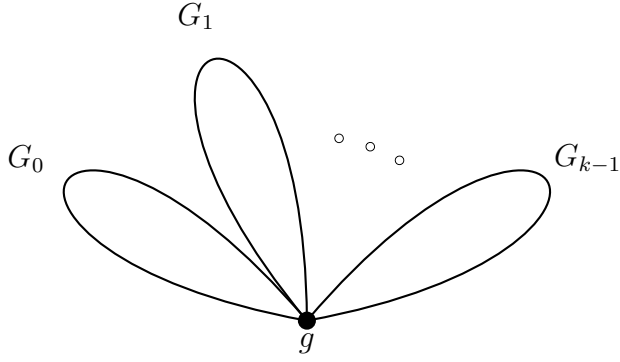
\begin{figure}[ht]
 \begin{center}

\psset{xunit=0.1\textwidth}
\psset{yunit=0.07\textwidth}
\begin{pspicture}(2,2)(8,7)

	\rput{ 60}(5,3){\psbezier(0,0)(-1,4)(1,4)(0,0) \rput[u]{*0}(0,3.5){$G_0$}}
	\rput{ 20}(5,3){\psbezier(0,0)(-1,4)(1,4)(0,0) \rput[u]{*0}(0,3.5){$G_1$}}
	\rput{-60}(5,3){\psbezier(0,0)(-1,4)(1,4)(0,0) \rput[u]{*0}(0,3.5){$G_{k-1}$}}
	\rput{-10}(5,3){\psdot[dotstyle=Bo](0,2)}
	\rput{-20}(5,3){\psdot[dotstyle=Bo](0,2)}
	\rput{-30}(5,3){\psdot[dotstyle=Bo](0,2)}

	\pscircle[fillstyle=solid, fillcolor=black](5,3){0.12}
	\uput[d](5,3){$g$}
\end{pspicture}

 \caption{A graph $G= \cup_{i=0}^{k-1} G_{i}$ where $G_{i} \cap G_{j} = \{g\}$ for all $i, j \in \{0, \ldots, k-1\}$, $i \neq j$.}
 \label{Amalgam}
 \end{center}
 \end{figure}
\end{lemma}

\begin{proof}
From Observation \ref{extensionOfEncoding}, without loss of generality we can assume that $l_{0} = \cdots = l_{k-1}= \max_{i}\{l_{i}\}$. Since we can rename the alphabets used in each coordinate of an encoding independently of the other coordinates, it is safe to assume that the vertex $g$ gets the encoding $\{0\}^{l_{0}}$ in every $\phi_{G_{i}}$. For all $ 0 \leq i \leq k-1$, let $b_{0}(i)$ denote the binary representation of $i$ using exactly $ \lceil \log k \rceil $ bits and $b_{1}(i)$ denote the bitwise complement of $b_{0}(i)$. The $l$-encoding $\phi _{G}$ of $G$ is as follows.

For all $i$, $0 \leq i \leq k-1$, for every $x \in V(G_{i} \setminus \{g\})$ 
\begin{eqnarray} \label{eqn}
 \phi_{G}(x) &=& \begin{cases}
	     \phi_{G_{i}}(x)  b_{0}(i)  & \textnormal{if $\phi_{G_{i}}(x)$ begins with $0$}\\ 
 	     \phi_{G_{i}}(x)  b_{1}(i)  & \textnormal{if $\phi_{G_{i}}(x)$ begins with $1$}\\ 
  \end{cases} \nonumber \\
\phi_{G}(g) &=& \{0\}^{l_{0}}\{2\}^{\lceil \log k \rceil}
\end{eqnarray}

We can verify that $\phi_{G}$ is a valid $l$-encoding of $G$ from the following argument.

Let $x, y \in V(G_{i} \setminus \{g\})$. If $\{x,y\} \in E(G_{i})$ then the first coordinates of $\phi_{G_{i}}(x)$ and $\phi_{G_{i}}(y)$ are different. Thus, the extra coordinates added to $\phi_{G_{i}}(x)$ and $\phi_{G_{i}}(y)$ to get $\phi_{G}(x)$ and $\phi_{G}(y)$ are complements of each other (by Equation (\ref{eqn})). If $\{x,y\} \not \in E(G_{i})$, then $\phi_{G_{i}}(x)$ and $\phi_{G_{i}}(y)$ agreed in some coordinate, say $t$. Hence, $\phi_{G}(x)$ and $\phi_{G}(y)$ also agree in the $t$-th coordinate.

Let $x \in V(G_{i} \setminus \{g\})$ and $y \in V(G_{j} \setminus \{g\})$ for some $i,j \in \{0, \ldots, k-1\}$, $i \neq j$. Note that, since $G_{i} \cap G_{j} = \{g\}$, $\{x,y\} \not \in E(G)$. If $\phi_{G_{i}}(x)$ and $\phi_{G_{i}}(y)$ agree in the first coordinate then $\phi_{G}(x)$ and $\phi_{G}(y)$ also agree in the first coordinate. If $\phi_{G_{i}}(x)$ begins with $0$ and $\phi_{G_{i}}(y)$ begins with $1$, then $\phi_{G}(x)=\phi_{G_{i}}(x) b_{0}(i)$  and $\phi_{G}(y)=\phi_{G_{j}}(y)  b_{1}(j)$. Since $i \neq j$, $b_{0}(i)$ and  $b_{1}(j)$ agree in some coordinate.

For any $i$, let $x \in V(G_{i} \setminus\{g\})$. If $\{g,x\} \not \in E(G_{i})$, then $\phi_{G_{i}}(g)$ and $\phi_{G_{i}}(x)$ agreed in some coordinate, say $t$. Hence, $\phi_{G}(g)$ and $\phi_{G}(x)$ also agree in the $t$-th coordinate. Otherwise, since $\phi_{G_{0}}(g)$ begins with $0$, $\phi_{G_{i}}(x)$ must begin with $1$. Thus, the extra coordinates added to $\phi_{G_{i}(x)}$ to get $\phi_{G}(x)$ are $b_{1}(i)$ while the extra coordinates added to $\phi_{G_{0}}(g)$ to get $\phi_{G}(g)$ are $\{2\}^{\lceil \log k \rceil}$. Therefore, $\phi_{G}(x)$ and $\phi_{G}(g)$ disagree in all coordinates.

It is easy to see from Equation \ref{eqn} that $\phi(G)$ is well-begun.
\end{proof}

\begin{theorem}\label{prodDimTrees}
For any forest $T$ on $n$ vertices, $pdim(T) \leq 1.441 \log n + 3$.
\end{theorem}

\begin{proof}
Let $V(T)=\{v_{0}, \ldots, v_{n-1}\}$, $f: V(T) \longrightarrow \{ 0, 1, \ldots , n-1\}$ be a bijection, and $f_{i}=f(v_{i})$. We use a divide and conquer strategy to prove the theorem. Let $C(T)$ denote the minimum $l$ such that there exists a \emph{well-begun} $l$-encoding of $T$. Let $C(n)=\max\{C(T) : T$ is a forest on at most n vertices$\}$.

Base Case: All possible forests with $|V(T)| \leq 3$ with there well-begun $3$-encodings are shown in Figure \ref{Encodings}. Thus, $C(3) \leq 3$. 

\begin{figure}[ht]
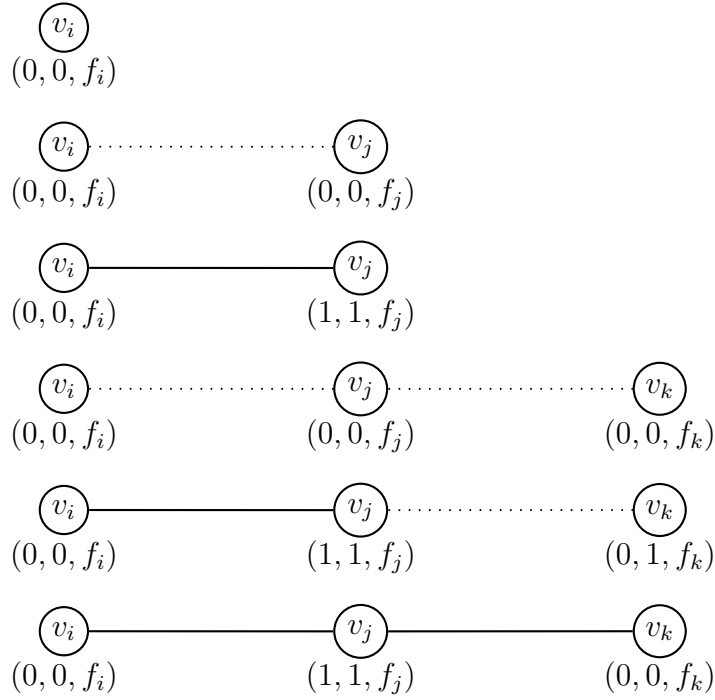

 \begin{center}
 \psmatrix[colsep=2.5cm,rowsep=0cm]
	[mnode=circle]$v_i$ \\
	$(0,0,f_i)$ 		\\
	&\\
	[mnode=circle]$v_i$ & [mnode=circle]$v_j$ 	\\
		\ncline[linestyle=dotted]{4,1}{4,2}
	$(0,0,f_i)$ 		& $(0, 0, f_j)$ 		\\
	&\\
	[mnode=circle]$v_i$ & [mnode=circle]$v_j$ 	\\ 
		\ncline{7,1}{7,2}
	$(0,0,f_i)$ 		& $(1, 1, f_j)$ 		\\
	&\\
	[mnode=circle]$v_i$ & [mnode=circle]$v_j$ 	& [mnode=circle]$v_k$ 	\\
		\ncline[linestyle=dotted]{10,1}{10,2}
		\ncline[linestyle=dotted]{10,2}{10,3}
	$(0,0,f_i)$ 		& $(0, 0, f_j)$ 		& $(0,0,f_k)$			\\
	&\\
	[mnode=circle]$v_i$ & [mnode=circle]$v_j$ 	& [mnode=circle]$v_k$ 	\\ 
		\ncline{13,1}{13,2}
		\ncline[linestyle=dotted]{13,2}{13,3}
	$(0,0,f_i)$ 		& $(1, 1, f_j)$ 		& $(0,1,f_k)$			\\
	&\\
	[mnode=circle]$v_i$ & [mnode=circle]$v_j$ 	& [mnode=circle]$v_k$ 	\\ 
		\ncline{16,1}{16,2}
		\ncline{16,2}{16,3}
	$(0,0,f_i)$ 		& $(1, 1, f_j)$ 		& $(0,0,f_k)$			\\
	&\\
 \endpsmatrix
 \caption{Well-begun $3$-encodings of the six forests with at most $3$ vertices. Each row depicts a single forest and dotted lines are non-edges.}
 \label{Encodings}
 \end{center}
 \end{figure}

Note that the third coordinate of each of the encodings is always a unique number associated with the vertex. This ensures injectivity of all the encodings that we get during the conquer steps.

Divide and Conquer: In our divide and conquer strategy, the divide operation corresponds to the two splitting operations of Lemma \ref{lem1} viz. \emph{($\epsilon$,2)-splitting} and \emph{($\epsilon$,3)-splitting} and the conquer operation corresponds to the amalgamation operation of Lemma \ref{lem2}. 

Choose $\epsilon =\frac{\sqrt{5}}{2}-1$. Let $\alpha = \frac{1}{2} + \epsilon$ and $\beta =\frac{1}{2} - \epsilon$. Note that $\alpha^{2}=\beta$. By Lemma \ref{lem1}, there exists either an \emph{($\epsilon$,2)-split vertex} or an \emph{($\epsilon$,3)-split vertex}, say $v \in V(T)$. 
If $v$ is an \emph{($\epsilon$,2)-split vertex}, then from Definition \ref{splittingDefn}, $T \setminus \{v\} = T_{1} \uplus T_{2}$ such that $|T_{1}|, |T_{2}| \leq \alpha n$. Let $T_{i}^{'}=T_{i} \cup \{v\}$, $i \in [2]$. Let $\phi_{T_{i}^{'}}$ be a \emph{well-begun} $l_{i}$-encoding of $T_{i}^{'}$, $i \in [2]$. Then by Lemma \ref{lem2}, there exists a \emph{well-begun} $l$-encoding $\phi_{T}$ of $T$ with $l=max\{l_{1},l_{2}\} + 1$. 
Similarly, if $v$ is an \emph{($\epsilon$,3)-split vertex}, then from Definition \ref{splittingDefn}, $T \setminus \{v\} = T_{1} \uplus T_{2} \uplus T_{3}$ such that $|T_{1}|, |T_{2}|, |T_{3}| \leq \beta n$. Let $T_{i}^{'}=T_{i} \cup \{v\}$, $i \in [3]$. Let $\phi_{T_{i}^{'}}$ be a \emph{well-begun} $l_{i}$-encoding of $T_{i}^{'}$, $i \in [3]$. Then by Lemma \ref{lem2}, there exists a \emph{well-begun} $l$-encoding $\phi_{T}$ of $T$ with $l=max\{l_{1},l_{2},l_{3}\} + 2$.

Therefore, the following recurrence relation holds.

\begin{eqnarray}
C(n) &\leq& \max \{C(\alpha n +1) + 1, C(\beta n +1) + 2 \} \nonumber \\ 
C(3) &\leq& 3 
\end{eqnarray}
Solving the recurrence: Let $X$ be an arbitrary leaf in the recurrence tree and let $P$ denote the path from the root to $X$. Let the number of \emph{($\epsilon$,2)-split} operations and \emph{($\epsilon$,3)-split} operations along $P$ be $k_{2}$ and $k_{3}$ respectively. Let $s_{i}$ be the size of the subgraph of $T$ to be conquered along $P$ after $i$ steps. Let $\gamma _{1},\ldots , \gamma _{k}$, $k=k_{2}+k_{3}$, be such that

\begin{equation}
 \gamma _{i} = \begin{cases}
	     \alpha  & \textnormal{if the $i$-th divide operation along $P$ is an \emph{($\epsilon$,2)-split} operation}\\
 	     \beta   & \textnormal{if the $i$-th divide operation along $P$ is an \emph{($\epsilon$,3)-split} operation} 
  \end{cases} 
\end{equation} 

Therefore, $s_{k} \leq (\prod_{j=1}^{k} \gamma _{j})n + \prod_{j=2}^{k} \gamma _{j} + \prod_{j=3}^{k} \gamma _{j} + \ldots + \prod_{j=k}^{k} \gamma _{j} +1$. Since $\gamma _{i} \leq \alpha$ for all $i$, $1 \leq i \leq k$, $s_{k} \leq (\prod_{j=1}^{k} \gamma _{j})n + \alpha ^{k-1} + \alpha ^{k-2} + \ldots + \alpha +1 \leq \alpha ^{k_{2}} \beta ^{k_{3}} n + \frac{1}{1- \alpha} \leq \alpha ^{k_{2}} \beta ^{k_{3}} n + 2.62$. Hence, $s_{k} \leq \lfloor \alpha^{k_{2}+2k_{3}}n + 2.62 \rfloor$. Note that $k_{2}+2k_{3}$ is the total cost of conquering (number of coordinates introduced by the amalgamation operation) incurred along $P$. Since $X$ is arbitrary, $C(n) \leq k_{2} + 2 k_{3} + C(s_{k})$.

Let $k_{2}+2k_{3} \geq 1.441 \log n$. Then $s_{k} \leq 3$. Hence, $C(n) \leq 1.441 \log n  + C(3) \leq 1.441 \log n  +3$. Therefore, $pdim(T) \leq 1.441 \log n + 3$.
\end{proof}
\section{Product Dimension of Bounded Treewidth Graphs}

\begin{definition} [Definition 1, \cite{chandran2007boxicity}]
A \emph{tree decomposition} of $G$ is a pair $(\{X_{i} : i \in I \},T )$, where $I$ is an index set, $\{X_{i} : i \in I \}$ is a collection of subsets of $V(G)$ and $T$ is a tree whose node set is $I$, such that the following conditions are satisfied:

\begin{enumerate}
 \item $\cup_{i \in I} X_{i}=V(G)$.
 \item $\forall \{u,v\} \in E(G), \exists i \in I$ such that $u,v \in X_{i}$.
 \item $\forall i,j,k \in I$ : if $j$ is on a path in $T$ from $i$ to $k$, then $X_{i} \cap X_{k} \subset X_{j}$.
\end{enumerate}

The \emph{width of a tree decomposition} $(\{ X_{i} : i \in I \},T)$ is $ \max \{|X_{i}| : i \in I\} -1$. The \emph{treewidth} of $G$, \emph{tw(G)}, is the minimum width over all tree decompositions of $G$. 

\end{definition}

Note that by a {\em rooted tree} we mean a tree with a vertex designated as the {\em root vertex}.


\begin{definition}[Definition 2, \cite{chandran2007boxicity}]
A \emph{normalized tree decomposition} of a graph $G$ is a triple $(\{X_{i} : i \in I \},r \in I,T )$ where $(\{X_{i} : i \in I \},T )$ is a tree decomposition of $G$ that additionally satisfies the following two properties:
\begin{enumerate}
\setcounter{enumi}{3}
 \item It is a rooted tree where the subset $X_{r}$ that corresponds to the root node $r$ contains exactly one vertex.
 \item For any node $i$, if $i^{'}$ is the child of $i$, then $|X_{i}^{'} - X_{i}|=1$ where, $X_{i^{'}} - X_{i}$ denoted the symmetric difference of $X_{i^{'}}$ and $X_{i}$.
\end{enumerate}
\end{definition}

\begin{lemma}[Lemma 3, \cite{chandran2007boxicity}]
For any graph $G$ there is a normalized tree decomposition with width equal to $tw(G)$.
\end{lemma}

\begin{lemma}[Splitting Lemma for Bounded Treewidth Graphs]\label{splitDegen}
Let $G$ be a graph on $n$ vertices with $tw(G)=t$ and a normalized tree decomposition $(\{X_{i}: i \in I\}, r \in I, T)$ of width $t$. Then there exists $l \in I$ such that $G \setminus X_{l} = G_{1} \uplus G_{2} \uplus G_{3}$ and $|G_{i}| \leq \frac{1}{2}(n-|X_{l}| + 1)$, $i \in [3]$, where $G_{1}, G_{2}$ and $G_{3}$ are subgraphs of $G$.
\end{lemma}

\begin{proof}
 For every $i$, let $D_{1}(i), \dots, D_{t}(i)$ be the components of $T \setminus \{i\}$ and let $C_{l}(i)$, $j \in [t]$, be the graphs induced by $(\cup_{j \in V(D_{l}(i))} X_{j}) - X_{i}$. Without loss of generality assume that $|C_{1}(i)| \geq \cdots \geq |C_{t}(i)|$. 
 
Let $c= \min \{|C_{1}(j)| : j \in I \}$ and $I^{'} = \{j \in I : |C_{1}(j)| =c\}$. Then choose $l \in I^{'}$ such that $|X_{l}| = \min \{|X_{j}| : j \in I^{'}\}$. We claim that, $|C_{1}(l)| \leq \frac{1}{2} (n-|X_{l}| + 1)$. For the sake of contradiction, assume that $|C_{1}(l)| > \frac{1}{2}(n - |X_{l}| + 1)$. Let $m \in N_{T}(l) \cap D_{1}(l)$. Then, since $T$ is a normalized tree decomposition $|X_{m}-X_{l}|=1$, therefore, the following two cases arise.

\begin{description}
 \item [Case 1] $X_{m}=X_{l} \cup \{v\}$ where $v \in V(G)$ \hfill \\
 If $D_{1}(m) \subset D_{1}(l)$, then $|C_{1}(m)| < |C_{1}(l)|$ because $C_{1}(m) = C_{1}(l) \setminus \{v\}$. Otherwise, $D_{1}(m)=T \setminus D_{1}(l)$ in which case $|C_{1}(m)| = |G \setminus (C_{1}(l) \cup X_{l})| = n -|C_{1}(l)| -|X_{l}| < n - \frac{1}{2}(n - |X_{l}| +1) - |X_{l}| =\frac{1}{2}(n - |X_{l}| -1) < |C_{1}(l)|$. In either case, $|C_{1}(m)| < |C_{1}(l)|$, contradicting the choice of $l$. 
 \item [Case 2] $X_{m}=X_{l} \setminus \{v\}$ where $v \in V(G)$ \hfill \\
If $D_{1}(m) \subset D_{1}(l)$, then $|C_{1}(m)| \leq |C_{1}(l)|$. If $|C_{1}(m)| < |C_{1}(l)|$, then $|C_{1}(m)|$ is not the minimum amongst all $|C_{1}(j)|$, $j \in I$ and if $|C_{1}(m)| = |C_{1}(l)|$ then, since $|X_{m}| < |X_{l}|$, the choice of $l$ is contradicted. On the other hand, if $D_{1}(m) = T \setminus D_{1}(l)$, then $|C_{1}(m)|= |G \setminus (C_{1}(l) \cup X_{m})|=n - |C_{1}(l)| - |X_{m}| < n -\frac{1}{2}(n - |X_{l}| + 1) - |X_{l}| +1 =\frac{1}{2}(n - |X_{l}| +1)$ again contradicting the choice of $l$.
\end{description}

Hence $C_{1}(l) \leq \frac{1}{2}(n - |X_{l}| + 1)$ i.e., $G \setminus X_{l}= C_{1}(l) \uplus \cdots \uplus C_{t}(l)$ such that $|C_{j}(l)| \leq \frac{1}{2}(n- |X_{l}| + 1)$ for all $j \in [t]$. 

Consider a partition $I_{1} \uplus \cdots \uplus I_{r} = [t]$ with minimum possible $r$ such that $|\cup_{j \in I_{i}} C_{j}(l)| \leq \frac{1}{2}(n-|X_{l}| +1)$ for all $i \in [r]$. Let $\cup_{j \in I_{i}} C_{j}(l) = H_{i}$ for all $i \in [r]$. Rename all $H_{i}$'s such that $|H_{1}| \geq \cdots \geq |H_{r}|$. We claim that for such a partition $r \leq 3$ because if $r \geq 4$ then $|\cup_{j = \lceil \frac{r}{2} \rceil + 1}^{r} H_{j}| \leq \frac{1}{2}(n-|X_{l}|)$ by the pigeonhole principle contradicting the choice of the partition $I_{1} \uplus \cdots \uplus I_{r}$. Set $G_{i}=H_{i}$ for $i \in [3]$ and we are done.
\end{proof}

\begin{lemma}[Amalgamation Lemma for General Graphs]\label{amalgamDegen}
 Let $G=G_{1} \cup G_{2} \cup G_{3}$ where $G_{1}, G_{2}$ and $G_{3}$ are graphs such that $G_{i} \cap G_{j} = S$ for all $i,j \in [3]$ and $i \neq j$. Let $G_{1}^{'}, G_{2}^{'}$ and $G_{3}^{'}$ be graphs such that $V(G_{i}^{'}) = V(G_{i})$ and $E(G_{i}^{'})=E(G_{i}) \cup \{\{v, v^{'}\} : v, v^{'} \in V(S)\}$ for all $i \in [3]$. Let $\phi_{G_{i}^{'}}$ be an $l_{i}$-encoding of $G_{i}^{'}$ for all $i \in [3]$ and $\phi_{S}$ be an $l_{s}$-encoding of $S$. Then we can construct an $l$-encoding of $G$, where 
\begin{eqnarray}
 l = \begin{cases}
      \max\{l_{1}, l_{2}, l_{3}\} + \max \{\chi(G \setminus S) + 1  , l_{s}\} & \textnormal{if $\chi(G \setminus S) =2$ or $6$} \\
      \max\{l_{1}, l_{2}, l_{3}\} + \max \{\chi(G \setminus S), l_{s}\} & \textnormal{otherwise}
     \end{cases}
\end{eqnarray}	

\begin{figure}[ht]
 \begin{center}

\psset{xunit=0.1\textwidth}
\psset{yunit=0.1\textwidth}
\begin{pspicture}(0,0.5)(10,3.5)

	
	\rput{90}(5,1){
		\psframe[fillstyle=solid, opacity=0.2, fillcolor=gray](-0.3,-0.3)(0.3,2)
		\rput{*0}(0,2.3){$G_1$}
	}
	\rput{0}(5,1){
		\psframe[fillstyle=solid, opacity=0.2, fillcolor=gray](-0.3,-0.3)(0.3,2)
		\rput{*0}(0,2.3){$G_2$}
	}
	\rput{-90}(5,1){
		\psframe[fillstyle=solid, opacity=0.2, fillcolor=gray](-0.3,-0.3)(0.3,2)
		\rput{*0}(0,2.3){$G_3$}
	}
	\rput(5,1){$S$}
\end{pspicture}
 \caption{A graph $G= \cup_{i=1}^{3} G_{i}$ where $G_{i} \cap G_{j} = S$ for all $i,j \in [3]$ and $i \neq j$}
  \label{amalgamGraphs}
 \end{center}
 \end{figure}
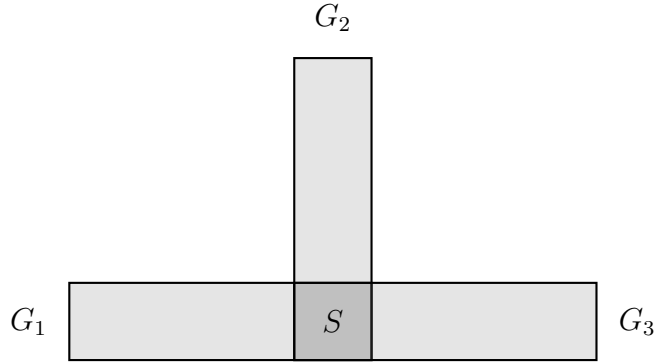

\end{lemma}

\begin{proof}
Without loss of generality we can assume that the alphabets used in $\phi_{S}$ are disjoint from the alphabets used in $\phi_{G_{i}^{'}}$ for all $i \in [3]$ and greater than $\chi(G)$, and also from Observation \ref{extensionOfEncoding}, let $l_{1} = l_{2} = l_{3} = \max\{l_{1}, l_{2}, l_{3}\}$. Let $V(G)=\{v_{0}, \ldots , v_{n-1}\}$, $f: V(G) \rightarrow \{0, \ldots , n-1\}$ be a bijection, and $f_{i}=f(v_{i})$. Let us rename the alphabets in each coordinate of $\phi_{G_{i}^{'}}$ such that all $v_{j} \in V(S)$ get the encoding as $(f_{j}, \ldots , f_{j})$ for all $i \in [3]$.

Let $c : V(G \setminus S) \rightarrow \{0, \ldots, \chi(G \setminus S) - 1\}$ be an optimal proper coloring of the vertices $V(G \setminus S)$.

 
Let 
\begin{eqnarray}
 t= \begin{cases}
     \chi(G \setminus S) + 1  & \textnormal{if $\chi(G \setminus S) =2$ or $6$} \\
     \chi(G \setminus S)  & \textnormal{otherwise}
    \end{cases}
\end{eqnarray}

By Theorem 4.3 in \cite{evans2000representations}, if we have two orthogonal Latin squares of order $t$, we can have a $t$-encoding for $3K_{t}$ and hence, for $3 K_{\chi(G \setminus S)}$ as well. Let the $j$-th vertex in the $i$-th copy of $3 K_{\chi(G \setminus S)}$ get the encoding $\phi_{K}(i,j)$ for all $i \in [3]$ and $j \in [\chi(G \setminus S)]$. Note that  $\phi_{K}(i,j)$ and $\phi_{K}(i,j^{'})$, $j \neq j^{'}$ disagree at all coordinates and $\phi_{K}(i,j)$ and $\phi_{K}(i^{'},j^{'})$, $i \neq i^{'}$, agree in at least one coordinate, for all $i, i^{'} \in [3]$ and $j,j^{'} \in [\chi(G \setminus S)]$. 
Let $m = \max\{t, l_{s}\}$. From Observation \ref{extensionOfEncoding}, let $\phi_{S}$ and $\phi_{K}(i,j)$ be $m$-encodings of $S$ and $3K_{\chi(G \setminus S)}$ respectively.
We construct an $l$-encoding of $G$, $\phi_{G}$, is as follows.
%
%
%

\begin{equation} \label{encodeAmal}
 \phi _{G}(x) = \begin{cases} 
		\phi_{G_{i}^{'}}(x)  \phi_{K}(i,c(x))  & \textnormal{if $x \in G_{i} \setminus S$} \\
		\phi_{G_{1}^{'}}(x) \phi_{S}(x)                            & \textnormal{$x \in S$}
		\end{cases}
\end{equation} 
We can verify that $\phi_{G}$ is a valid encoding of $G$ from the following argument.
Let $x,y \in V(G_{i} \setminus S)$. If $\{x,y\} \not \in E(G)$, then $\{x,y\} \not \in E(G_{i}^{'})$. Therefore, $\phi_{G_{i}^{'}}(x)$ and $\phi_{G_{i}^{'}}(y)$ agree in some coordinate, say $g$ and thus, $\phi_{G}(x)$ and $\phi_{G}(y)$ also agree in the $g$-th coordinate. If $\{x,y\} \in E(G)$, then $\{x,y\} \in E(G_{i}^{'})$. Hence, $\phi_{G_{i}^{'}}(x)$ and $\phi_{G_{i}^{'}}(y)$ do not agree in any coordinate and since, $c$ is a proper coloring of $G \setminus S$, $c(x) \neq c(y)$. Thus, $\phi_{K}(i, c(x))$ and $\phi_{K}(i, c(y))$ do not agree in any coordinate. Therefore, $\phi_{G}(x)$ and $\phi_{G}(y)$ do not agree in any coordinate.
 
Let $x \in V(G_{i} \setminus S)$ and $y \in V(G_{i^{'}} \setminus S)$, $i\neq i^{'}$. Note that $\{x,y\} \not \in E(G)$. Since $\phi_{K}(i,c(x))$ and $\phi_{K}(i^{'}, c(y))$, $i \neq i^{'}$ agree in some coordinate, say $g$, $\phi_{G}(x)$ and $\phi_{G}(y)$ will agree in the $(l_{1} + g)$-th coordinate.

For any $i$, let $x \in V(G_{i} \setminus S)$ and $y \in V(S)$. Since $\phi_{S}(y)$ uses new alphabets greater than $\chi_{G}$, $\phi_{G}(x)$ and $\phi_{G}(y)$ agree in some coordinate if and only if $\phi_{G_{i}^{'}}(x)$ and $\phi_{G_{1}^{'}} (y)$ ($= \phi_{G_{i}^{'}}(y)$) agree in some coordinate. 

For $x,y \in V(S)$, if $\{x,y\} \in E(G)$, then since $\phi_{G_{1}^{'}}(x)=(f(x), \ldots, f(x))$, $\phi_{G_{1}^{'}}(y)=(f(y), \ldots, f(y))$ where $f$ is a bijection and $\phi_{S}(x)$ and $\phi_{S}(y)$ disagree in all coordinates, $\phi_{G}(x)$ and $\phi_{G}(y)$ disagree in all coordinates. If $\{x,y\} \not \in E(G)$, then $\{x,y\} \not \in E(S)$, Thus, $\phi_{S}(x)$ and $\phi_{S}(y)$ agree in some coordinate, say $g$ and therefore, $\phi_{G}(x)$ and $\phi_{G}(y)$ agree in the $(l_{1} + g)$-th coordinate. 
\end{proof}

\begin{theorem} \label{thmTreeWidth}
For any graph $G$ on $n$ vertices and $tw(G)=t$, $pdim(G) \leq (t+2)(\log n + 1)$.
\end{theorem}

\begin{proof}
We use a divide and conquer strategy to prove the theorem. Let $C_{t}(n)=\max \{ pdim(G) : G$ is a graph on at most $n$ vertices and $tw(G) \leq t \}$.

Base Case: By Theorem 4.3 in \cite{lovasz1980product}, $C_{t}(t+3) = t+2$.

Divide and Conquer: In our divide and conquer strategy, the divide operation corresponds to the splitting operation of Lemma \ref{splitDegen} and the conquer operation corresponds to the amalgamation operation of Lemma \ref{amalgamDegen}.

By Lemma \ref{splitDegen}, for a graph $G$ on $n$ vertices with $tw(G)=t$ and a normalized tree decomposition $(\{X_{i} : i \in I\}, r \in I, T)$ of width  $t$, there exists  $l \in I$ such that $G \setminus X_{l} = G_{1} \uplus G_{2} \uplus G_{3}$,  $|G_{i}| \leq \frac{1}{2}(n-|X_{l}| + 1)$, $i \in [3]$. Let $G_{i}^{'}=G_{i} \cup G[X_{l}]$ for all $i \in [3]$. Therefore, $|G_{i}^{'}| \leq \frac{1}{2}(n-|X_{l}|+1) + |X_{l}|=\frac{1}{2}(n+|X_{l}|+1)$ for all $i \in [3]$. Let $\alpha = \frac{1}{2}$ and $\beta = \frac{|X_{l}|+1}{2}$. Hence, $|G_{i}^{'}| \leq \alpha n + \beta$ for all $i \in [3]$.

Let $S = G[X_{l}]$.Note that $G_{i}^{'} \cap G_{j}^{'} =S$ for all $i,j \in [3]$ and $i \neq j$, and  $G=G_{1}^{'} \cup G_{2}^{'} \cup G_{3}^{'}$. Let $G_{1}^{''}, G_{2}^{''}, G_{3}^{''}$ be graphs such that $V(G_{i}^{''}) = V(G_{i}^{'})$ and $E(G_{i}^{''})=E(G_{i}^{'}) \cup \{\{v, v^{'}\}: v, v^{'} \in V(S)\}$ for all $i \in [3]$ (note that $|G_{i}^{''}| \leq \alpha n + \beta,$ $i \in [3]$). Let $\phi_{G_{i}^{''}}$ is an $l_{i}$-encoding of $G_{i}^{''}$ for all $i \in[3]$ and $\phi_{S}$ be an $l_{s}$-encoding of $S$. Then, by Lemma \ref{amalgamDegen}, we can construct an $l$-encoding of $G$ where 
\begin{eqnarray}
 l = \begin{cases}
      \max\{l_{1}, l_{2}, l_{3}\} + \max\{\chi(G \setminus S) + 1 , l_{s}\} & \textnormal{if $\chi(G \setminus S) =2$ or $6$} \\
      \max\{l_{1}, l_{2}, l_{3}\} + \max\{\chi(G \setminus S), l_{s}\} & \textnormal{otherwise}.
     \end{cases}
\end{eqnarray}
 Since $G$ is a graph with $tw(G)=t$, $\chi(G) \leq t+1$(Theorem 6, \cite{lu2007note}), and hence $\chi(G \setminus S) \leq t+1$. Also, since $|V(S)| \leq t+1$, by Theorem 4.3 \cite{lovasz1980product}, $l_{s} \leq t+1$.
Therefore,

\begin{eqnarray}
 l \leq \begin{cases}
      \max\{l_{1}, l_{2}, l_{3}\} + \max\{t+2, t+1\} & \textnormal{if $\chi(G \setminus S) =2$ or $6$} \nonumber \\
      \max\{l_{1}, l_{2}, l_{3}\} + \max\{t+1, t+1\} & \textnormal{otherwise} \nonumber \\
     \end{cases}
\end{eqnarray}
Hence,
\begin{eqnarray}     
 l \leq \max\{l_{1}, l_{2}, l_{3}\} + t+2
\end{eqnarray}
  
Let $G^{'}$ be the graph such that $V(G^{'}) = V(G)$ and $E(G^{'}) = E(G) \cup \{\{v, v^{'}\} : v , v^{'} \in V(S)\}$. Note that $(\{X_{i} :  i \in I\}, r \in I, T)$ is a tree decomposition for $G^{'}$ too since all the new edges added are between the vertices of the same node. Also since $G \subset G^{'}$, $tw(G) \leq tw(G^{'})$. Hence, $tw(G^{'}) = t$ and thus, $tw(G_{i}^{''}) \leq t$ (since $G_{i}^{''} \subset G^{'}$) for all $i \in [3]$.  
  
Therefore, the following recurrence relation holds.
\begin{eqnarray}
C_{t}(n) &\leq& 
    C_{t}(\alpha n + \beta) + t+2 \nonumber \\ 
C_{t}(t+3) &\leq& t+2 
\end{eqnarray}
Solving the recurrence: Let $X$ be an arbitrary leaf in the recurrence tree and let $P$ denote the path from root to $X$. Let the number of divide operations along $P$ be $d$. Let $s_{j}$ be the size of the subgraph of $G$ to be conquered along $P$ after $j$ steps. 

Therefore, $s_{d} \leq \alpha^{d}n + \alpha^{d-1} \beta + \alpha^{d-2} \beta + \cdots + \alpha \beta + \beta \leq \alpha^{d}n + \frac{\beta}{1-\alpha} = \alpha^{d}n +|X_{l}|+1 \leq \alpha^{d}n +t+2$ (since $|X_{l}| \leq t+1$). Hence, $s_{d} \leq \lfloor \alpha^{d}n + t+2 \rfloor$. Note that the total cost of conquering incurred along $P$ is $(t+2)d$. 


Let $d \geq \log n$. Then $s_{d} \leq t+3$. Hence, $C_{t}(n) \leq (t+2)\log n + C_{t}(t+3) \leq (t+2) \log n + t+2 = (t+2)(\log n +1)$. Therefore, $pdim(G) \leq (t+2)(\log n + 1)$.
\end{proof}
\section{Product Dimension of $k$-degenerate Graphs}
Let $v_{1}, \ldots, v_{n}$ be an ordering of the vertex set of $G$ such that $|N(v_{i}) \cap \{v_{j} : j< i\}| \leq k$. If for a graph $G$ such an ordering exists, then the graph $G$ is called $k$-degenerate and the set $N_{G}(v_{i}) \cap \{v_{j} : j< i\}$ is called the set of backward neighbors of $v_{i}$.

 \begin{theorem} \label{thmDegenerate}
 For every $k$-degenerate graph $G$, $pdim(G) \leq \lceil 8.317 k \log n \rceil + 1$.
 \end{theorem}
 \begin{proof}
Recall that the product dimension of a graph $G$ is the minimum number of proper colorings of $G$ such that any pair of non-adjacent vertices get the same color in at least one of the colorings and not in all of them.
 
We use probabilistic arguments to prove the theorem. Let us describe a random coloring procedure using $3k$ colors for the vertices of $G$. Let $C=[3k]$ be the set of colors. We color the vertices starting from $v_{1}$ such that any vertex $v_{i}$ is assigned a color uniformly at random from the set $C \setminus C_{i}$, where $C_{i}$ is the set of colors used by the backward neighbors of $v_{i}$. Note that $0 \leq |C_{i}| \leq k$.
Repeat this procedure $p$ times to get $p$ random colorings. This procedure ensures that colorings are proper.

For $\{v_{i}, v_{j}\} \not \in E(G)$, let us calculate the probability that both $v_{i}$ and $v_{j}$ get the same color in a particular coloring. Let $C^{'}=C \setminus (C_{i} \cup C_{j})$. Then the probability that both $v_{i}$ and $v_{j}$ get the same color in a particular coloring is equal to the probability that $v_{i}$ chooses a color from the set $C^{'}$ and $v_{j}$ chooses the same color as chosen by $v_{i}$ from the set $C^{'}$. Hence,the probability that $v_{i}$ and $v_{j}$ get the same color in a particular coloring $= \frac{|C^{'}|}{|C| - |C_{i}|} \frac{|C^{'}|}{|C| -|C_{j}|} \frac{1}{|C^{'}|} \geq  \frac{|C^{'}|}{|C^{'}| + |C_{j}|} \frac{|C^{'}|}{|C^{'}| +|C_{i}|} \frac{1}{|C^{'}|} (\because |C^{'}|\geq|C|-|C_{i}|-|C_{j}|)$. Note that $0 \leq |C_{i}| \leq k \leq |C^{'}|$ hence $|C_{i}| \leq |C^{'}|$ and $|C_{j}| \leq |C^{'}|$. Therefore, $\frac{|C^{'}|}{|C^{'}|+|C_{i}|} , \frac{C^{'}|}{|C^{'}|+|C_{j}|} \geq \frac{1}{2}$. Also, $\frac{1}{|C^{'}|} \geq \frac{1}{|C|} = \frac{1}{3k}$ ( since $|C^{'}| \leq |C|$). Thus, the probability that $v_{i}$ and $v_{j}$ get the same color in a particular coloring $ \geq \frac{1}{6k}$. The probability that $v_{i}$ and $v_{j}$ get different colors in a particular coloring $ \leq (1-\frac{1}{6k})$. Therefore, the probability that $v_{i}$ and $v_{j}$ get different colors in all the $p$ colorings $ \leq (1-\frac{1}{6k})^{p} \leq e^{\frac{-p}{6k}}$. Hence, the probability that all pairs of non-adjacent vertices get different colors in all the $p$ colorings $ < n^{2}e^{\frac{-p}{6k} }$. If $p \geq 12k \ln n = 8.317 k \log n$, $n^{2} e^{\frac{-p}{6k}} \leq 1$. Thus, if $p = \lceil 8.317 k \log n \rceil$, then every pair of non-adjacent vertices in the graph gets the same color in at least one of the $p$ colorings described above. There might exist a case when a pair of non-adjacent vertices get the same color in all the colorings in which case we also consider a $(p+1)$-th coloring where all vertices get a unique color. Thus $pdim(G) \leq \lceil 8.317 k \log n \rceil + 1$.
\end{proof}

\bibliographystyle{plain}

\begin{thebibliography}{10}

\bibitem{alles1986dimension}
P.~Alles.
\newblock Dimension of amalgamated graphs and trees.
\newblock {\em Czechoslovak Mathematical Journal}, 36(3):393--416, 1986.

\bibitem{alon1986covering}
N.~Alon.
\newblock Covering graphs by the minimum number of equivalence relations.
\newblock {\em Combinatorica}, 6(3):201--206, 1986.

\bibitem{bose1960further}
R.C. Bose, S.S. Shrikhande, and E.T. Parker.
\newblock Further results on the construction of mutually orthogonal {Latin}
  squares and the falsity of {Euler's} conjecture.
\newblock {\em Canad. J. Math}, 12:189--203, 1960.

\bibitem{chandran2007boxicity}
L.S. Chandran and N.~Sivadasan.
\newblock Boxicity and treewidth.
\newblock {\em Journal of Combinatorial Theory, Series B}, 97(5):733--744,
  2007.

\bibitem{cooper10}
Jeffery~R. Cooper.
\newblock Product dimension of a random graph.
\newblock Masters thesis, Miami University, Oxford, Ohio, 2010.

\bibitem{eaton1996graphs}
N.~Eaton and V.~R{\"o}dl.
\newblock Graphs of small dimensions.
\newblock {\em Combinatorica}, 16(1):59--85, 1996.

\bibitem{erdos1989representations}
P.~Erd{\"o}s and A.B. Evans.
\newblock Representations of graphs and orthogonal {Latin} square graphs.
\newblock {\em Journal of Graph Theory}, 13(5):593--595, 1989.

\bibitem{evans2007representations}
A.B. Evans.
\newblock Representations of disjoint unions of complete graphs.
\newblock {\em Discrete mathematics}, 307(9):1191--1198, 2007.

\bibitem{evans2000representations}
A.B. Evans, G.~Isaak, and D.A. Narayan.
\newblock Representations of graphs modulo n.
\newblock {\em Discrete Mathematics}, 223(1):109--123, 2000.

\bibitem{furedi2000prague}
Z.~F\"{u}redi.
\newblock On the {Prague} dimension of {Kneser} graphs.
\newblock {\em Numbers, information, and complexity}, page 125, 2000.

\bibitem{kantor10}
Ida Kantor.
\newblock Graphs, codes, and colorings.
\newblock {Ph.D} thesis, Graduate College of the University of Illinois at
  Urbana-Champaign, 2010.

\bibitem{koster02}
A.M.C.A. Koster, H.L. Bodlaender, and S.P.M. Van~Hoesel.
\newblock Treewidth: computational experiments.
\newblock Research Memoranda 0001, Maastricht : METEOR, Maastricht Research
  School of Economics of Technology and Organization, 2001.

\bibitem{lovasz1980product}
L.~Lov{\'a}sz, J.~Ne\v{s}et\v{r}il, and A.~Pultr.
\newblock On a product dimension of graphs.
\newblock {\em Journal of Combinatorial Theory, Series B}, 29(1):47--67, 1980.

\bibitem{lu2007note}
C.~Lu.
\newblock A note on lower bounds of treewidth for graphs.
\newblock In {\em International Mathematical Forum}, volume~2, pages
  2893--2898, 2007.

\bibitem{nesetril1978simple}
J.~Ne\v{s}et\v{r}il and V.~R{\"o}dl.
\newblock A simple proof of the Galvin-Ramsey property of the class of all
  finite graphs and a dimension of a graph.
\newblock {\em Discrete Mathematics}, 23(1):49--55, 1978.

\bibitem{poljak1981dimension}
S.~Poljak and A.~Pultr.
\newblock On the dimension of trees.
\newblock {\em Discrete Mathematics}, 34(2):165--171, 1981.

\end{thebibliography}

\end{document}